\documentclass{amsart}
\usepackage{graphicx}
\usepackage{amssymb}
\usepackage{graphicx}
\usepackage{amsfonts}
\usepackage{amsmath}
\usepackage{amsthm}
\usepackage{esint}
\usepackage{mathabx}
\usepackage{mathtools}
\usepackage{nccmath}
\usepackage[colorlinks]{hyperref}

\newtheorem{thm}{Theorem}[section]

\newtheorem{lem}[thm]{Lemma}
\newtheorem{prop}[thm]{Proposition}
\theoremstyle{definition}

\theoremstyle{remark}
\newtheorem{rem}[thm]{Remark}
\theoremstyle{remark}
\newtheorem{ex}[thm]{Example}

\renewcommand{\chi}{{ 1}}

\begin{document}

\title[Low Codimensional singularities]{Solutions to Monge-Ampère Equations with Low Codimensional singularities}

\author{Arghya Rakshit}
\address{Department of Mathematics, University of Toronto}
\email{\tt arghya.rakshit@utoronto.ca}
\author{Aranya Sen}
\address{Department of Mathematics, UC Irvine}
\email{\tt aranyas1@uci.edu}
%AMS Subject Classification

% Keywords and Phrases
\date{\today}

% ----------------------------------------------------------------

%We construct solutions to Monge--Amp\`ere equations whose Monge--Amp\`ere
%measures contain singular components supported on sets of low codimension.
%We also study the regularity of these solutions. To motivate our construction,
%we present examples arising from optimal transport where the potentials of the
%optimal transport maps satisfy Monge--Amp\`ere equations similar to the ones we study.

\begin{abstract}
We construct solutions to Monge--Amp\`ere equations whose Monge--Amp\`ere measures contain singular components supported on low codimensional sets. We also study the regularity of such solutions. To motivate our construction, we present examples arising from optimal transport where the potential of optimal transport maps satisfy Monge–Ampère equations similar to the ones we study.

\end{abstract}
\maketitle

% ----------------------------------------------------------------

%Monge–Ampère

%%%%%%%%%%%%%%
\section{Introduction}\label{intro}
Solutions of Monge–Ampère equations with singular Monge–Ampère measures arise in many settings, such as optimal transport and mirror symmetry (see \cite{CL2}, \cite{H1}, \cite{HJMM}, \cite{JX1}, \cite{L1}, \cite{Lo1}, and \cite{LYZ1}). They also show up in the recent work of Jin, Tu and Xiong (see \cite{JTX1}, and \cite{JTX2}), while trying to study sharp estimates of the exponent in the Alexandrov estimate. Motivated by questions arising in mirror symmetry, the construction of solutions to Monge–Ampère equations with Monge-Ampère measure having isolated singularities was studied in \cite{M1} and \cite{MR1}. It was observed that singularities in solutions to such Monge–Ampère equations can also arise away from the point masses. In particular it was shown that solutions with singularities of dimension at most $n/2$ can be constructed. The construction in \cite{MR1} was carried out by solving an obstacle problem where the obstacle was supported only at a finite number of points. In this article we investigate solutions of an obstacle problem where the obstacle is supported on a low codimension set.  We construct solutions of Monge-Amp\`ere equations with Monge–Ampère measures supported on sets of low codimension and study their smoothness away from these low codimensional sets. In particular, the solutions we construct are for codimensions less than $2$ away from the support of the obstacle. We also study the singularities when codimension is more than $2$. These constructions are inspired from optimal transport of nonconvex domains. We also explore this connection through a series of examples.  We describe our results when these sets are a line segment, boundary of a convex polytope,  unions of line segments, and boundary of convex sets with $C^1$ boundary.  \\

First we fix some notations. Let $n\geq 2$ and $1\leq k\leq n-1$. We write coordinates as $(x,y)\in \mathbb{R}^{n-k}\times \mathbb{R}^k$. By $B_R$, we mean an open ball of radius $R$ centered at the origin in $\mathbb{R}^n$. We denote by $\mathcal{H}^k$ the $k$-dimensional Hausdorff measure in $\mathbb{R}^n$.   For a detailed discussion on Hausdorff measure, see \cite{EG}.  For a convex function $f$ defined on a subset of $\mathbb{R}^n$, $\partial f(x)$ denotes the subdifferential of $f$ at $x$. For a set $S\subset \mathbb{R}^n$, $\partial S$ denotes the boundary of $S$ in $\mathbb{R}^n$. So the notation $\partial$ should be clear depending on context. For $a>0$, define 
\[
S_a = \{0\}^{n-k}\times (-a,a)^k.\]
We also assume $P$ to be a convex polytope whose vertices lie on or outside a sphere $B_a$, and for any $0<\alpha<1$ define $\Lambda_{\alpha}:= \partial P\cap B_{\alpha a}$. Furthermore, let $P[k]$ denote the $k$-skeleton of $P$. We prove the following theorems:
\begin{thm}\label{line}
    For any $a>0$ and $\alpha\in (0,1)$,
    there exists a convex function $u\in C^\infty(\mathbb{R}^n\setminus S_a)$, and a non-negative function $f \in L^{\infty}(S_a)$, that solves
    $$\det(D^2u) = 1+f\mathcal{H}^{k}|_{S_a}$$
    such that $f \geq c>0$ on $S_{\alpha a}$, where $c$ is a constant that depends on $a$ and $n$ .
\end{thm}

\begin{thm}\label{polytope}
    For $k = 1$ or $2$ and $0<\alpha<1$, there exists a convex function $u\in C^\infty(\mathbb{R}^n\setminus P[n-k])$ that solves 
    $$\det(D^2u) = 1+f\mathcal{H}^{n-k}|_{P[n-k]},$$
     where $f\in L^{\infty}(\partial P)$, $f \geq c>0$ on $ P[n-k]\cap \Lambda_\alpha$ and $c$ depends on $a$ and $n$.
\end{thm}
\begin{rem}\label{rem:irragularpolytope}
We can modify the proof of Theorem \ref{polytope}  to construct solution $u\in C^{\infty}(\mathbb{R}^n\setminus (P[n-k]\cup \Sigma_\delta))$  for $k=1,2$ and $\delta>0$. Here $\Sigma_\delta$ is the singular set of $u$ of dimension $n-k-\delta$  inside the polytope (see Remark \ref{rem:justification1}).
\end{rem}
\begin{thm}\label{cross}
    Let dimension $n=2,3$. Let $L_1,L_2,\cdots, L_k$ be line segments starting from the origin. We can construct a solution
    $$\det(D^2u) = 1+c_0\delta_0+f\sum_{i=1}^k\mathcal{H}^1|_{L_i}$$ 
    such that $u\in C^{\infty}(\mathbb{R}^n\setminus \{L_1,\cdots,L_k\})$ and $f\geq c_1$ on $\cup_{i=1}^k \alpha L_i$, for any $0<\alpha<1$ and for some positive constants $c_0,c_1$ depending on lengths of line segments and dimension.
\end{thm}

\begin{thm}\label{smooth_boundary}
    Let $\Omega$ be an open convex set with $C^1$ boundary. There exists a convex function, $u\in C^\infty(\mathbb{R}^n\setminus \partial \Omega)$, and a non-negative function $f \in L^{\infty}(\partial \Omega)$, that solves,
    $$\det(D^2u) = 1+f\mathcal{H}^{n-1}|_{\partial \Omega}$$
    such that $f \geq c_{\Omega}$ on $\partial \Omega$ for some positive constant $c_\Omega$ depending on $\Omega$ and dimension $n$.
\end{thm}

\begin{rem}
    In the theorems above, the solutions we construct are asymptotic to some quadratic polynomial at infinity. Fixing the quadratic polynomial gives us a unique solution. This follows from Theorem \ref{CL3_exterior}.
\end{rem}

The reason we work with $k=1 \text{ and }2$ in Theorem \ref{polytope}  and $n=2\text{ and }3$ in Theorem \ref{cross} is to avoid the development of singularities of the solutions away from the support of the singular measures i.e. to avoid interaction between singular measures. A natural question to ask is whether these singular measures interact if we further relax the assumption to $k\geq 3$ in Theorem \ref{polytope}. The answer to this question depends crucially on the precise geometry of the support of the singular measures. For technical reasons, we delay stating a general result to Section \ref{interaction_section}. For now, we state an abridged version of the discussion in the form of a theorem: 

\begin{thm}\label{simple_interatcion}
    Let $P_n$ be a regular convex polytope in $\mathbb{R}^n$. For $k\geq 3$ and $\alpha \in (0,1)$, there exists a convex function $u\in C^{\infty}(\mathbb{R}^n\setminus (P[n-k]\cup \Sigma))$ that solves,
    $$\det(D^2u) = 1+f\mathcal{H}^{n-k}|_{P[n-k]}$$
    where $f\in L^{\infty}(\partial P)$, $f \geq c>0$ on $ P[n-k]\cap \Lambda_\alpha$ and $c$ depends on $a$ and $n$. The singular set $\Sigma$ away from $P[n-k]$ is non-empty and its structure can be determined. Moreover the dimension of $\Sigma$ is at least $n-k+1$.
\end{thm}
For a proof of this theorem see the discussion following Theorem \ref{vol_obs}.

\begin{rem}
    By using a variation of our method we can push the codimension of the above result to $k=2$ with the dimension of the singular set being $n-1-\delta$ for any small $\delta>0$. By the partial regularity results in \cite{M2} the dimension of the singular sets can't be improved. This gives a large new family of solutions to Monge-Ampere equations with right hand side equal to 1 which have singular sets arbitrarily close to $n-1$ (see remark \ref{sharp_sing}).
\end{rem}

Next, we compare our results in this article with \cite{M1}, and \cite{MR1}. 
%Say solving an obstacle problem is core to for all these works.
In contrast to \cite{M1}, which focuses on dimensions $n=3,4$, and \cite{MR1}, where singular sets of dimension at most $n/2$ are constructed, the results in this article apply in arbitrary dimensions and yield higher-dimensional singularities. This is made possible because we allow our obstacles to live on low codimensional subsets which forces singularity of solutions on these sets. This however makes the study of regularity away from those sets more involved. Additionally, to ensure the Monge–Ampère measures associated with the solutions to our obstacle problem don't develop Dirac singularities, we have to construct our obstacles carefully.  \\

%Write a line saying one can study the singularity like the paper of Liu(Open question).

The paper is organized as follows: In Section \ref{priliminaries}, we provide some basic definitions regarding solutions to Monge–Ampère equations, solve an obstacle problem for Monge–Ampère equations with a general obstacle, and  construct some barriers that we use later to prove the theorems above. 
In Section \ref{examples}, we motivate the connection between optimal transport and the type of singular measures we construct. In particular, we give several examples of Monge–Ampère measure with singularities arising in optimal transport.  Moreover, we also motivate a new connection between optimal transport and Signorini problem. In Section \ref{proof}, we prove Theorems \ref{line}, \ref{polytope}, \ref{cross}, and \ref{smooth_boundary}. Section \ref{interaction_section} is dedicated towards interaction of singular measures. Finally, in Section \ref{open}, we discuss some open problems motivated by our work.

\section{Preliminaries}\label{priliminaries}

In this section we define Alexandrov solution to  Monge–Ampère equation, define sections for convex functions, recall a few theorems from the literature, solve an obstacle problem related to Monge–Ampère equation,  and construct suitable barriers that will be used to prove the theorems above.

\subsection{Basics of Monge–Ampère equation}
To a convex function $v$ on a domain $\Omega \subset \mathbb{R}^n$ we associate a Borel measure $\mathcal{M}v$ on $\Omega$, called the Monge-Amp\`{e}re measure of $v$. It satisfies
$$\mathcal{M}v(E) = |\partial v(E)|$$
for any Borel set $E \subset \Omega$, where $\partial v$ denotes the subgradient of $v$. When $v \in C^2$, we have $\mathcal{M}v = \det D^2v \,dx.$
Given a Borel measure $\mu$ on $\Omega$, we say that $v$ is an Alexandrov solution to the Monge-Amp\`{e}re equation $\det D^2v = \mu$
if  $\mathcal{M}v = \mu.$ In this article, when we say $u$ satisfies a  Monge–Ampère equation, we mean it satisfies the equation in Alexandrov sense.\\

    Let $u:\mathbb{R}^n \to \mathbb{R}$ be a convex function. Given $x\in \mathbb{R}^n$, $p\in \partial u(x)$ and $h>0$, define the section of height $h$ at $x$ by
\[
S_h(x,p):=\left\{y\in \mathbb{R}^n:\ u(y)<u(x)+\langle p,y-x\rangle+h\right\}.
\]
We use this definition in the proof of Theorem \ref{polytope}. Next we discuss two results that we use in Section \ref{proof}. The first Theorem is due to Caffarelli:
\begin{prop}[Caffarelli, \cite{C0}]\label{no_extremal}
Let $u$ be a convex Alexandrov solution of
\[
0<\lambda\ \leq \det(D^2 u) \leq \Lambda<\infty
\]
in a domain $\Omega \subset \mathbb{R}^n$. Let $L$ be a supporting affine function to $u$, and consider the contact set
\[
\Sigma := \{x \in \Omega : u(x) = L(x)\}.
\]
Then $\Sigma$ has no extremal points in $\Omega$. 
\end{prop}

\noindent 
The second theorem is due to Caffarelli and Li:

\begin{thm}[Caffarelli--Li, {\cite[Corollary 1.3]{CL3}}]\label{CL3_exterior}
Let $O \subset \mathbb{R}^n$ be a bounded open convex set, and let $u$ be a locally convex Alexandrov solution of
\[
\det(D^2 u)=1 \quad \text{in } \mathbb{R}^n \setminus O.
\]
Then
\[
u \in C^\infty(\mathbb{R}^n \setminus \overline{O}).
\]
\end{thm}

%% Add Caffarelli-Yuan result as well
%% Add Boundary maps to boundary
\subsection{Obstacle Problem}\label{obstacles}
At the heart of our proof lies an obstacle problem. The data are a bounded strictly convex domain $U \subset \mathbb{R}^n$, boundary data $\varphi \in C\left(\overline{U}\right)$, an obstacle $g : \overline{U} \rightarrow \mathbb{R} \cup \{+ \infty\}$ which is lower semicontinuous and satisfies $g > \varphi$ on $\partial U$, and a finite Borel measure $\mu$ on $U$. We define the class of functions $\mathcal{F}$ by

\begin{equation}\label{class}
    \mathcal{F} := \left\{v : v \in C\left(\overline{U}\right) \text{ convex},\, v \leq g \text{ in } U,\, v|_{\partial U} \leq \phi,\, Mv \geq \mu\right\}.
\end{equation}

\noindent
It is shown in \cite[Proposition 2.1]{MR1},
\begin{lem}\label{obstacle_problem}
	The set $\mathcal{F}$ is non-empty, the function
	$$u := \sup_{\mathcal{F}} v$$ 
	belongs to $\mathcal{F}$, and 
	$$Mu = \mu \text{ in } \{u < g\} \cap U.$$
\end{lem}

We will primarily use the obstacle $g = g_{n,k,\alpha}$, to be defined below. For  $0<\alpha<1$, define $g_1$ to be
  $$
       g_1(x) = \left\{
	\begin{array}{ll}
		\frac{x^2}{2},  & \mbox{if } x\in \{|x|\leq \alpha\}; \\
		  \varphi_\alpha, & \mbox{if } x \in \{\alpha<|x| \leq 1\};\\
          \infty, & \text{otherwise},
	\end{array}
\right.
  $$
   where $\varphi_\alpha':(\alpha,1]\to \mathbb{R}$ is an increasing function  with $\varphi_\alpha' \to \infty$ as $x\to 1$ , $\varphi_\alpha \leq 5$ and $g_1\in C^\infty((-1,1))$. Now for $x\in \mathbb{R}^n$ define $g_{n,\alpha}$ to be
   \begin{equation}
       g_{n,\alpha}(x) = g_1(|x|).
   \end{equation}
   For $(x,y)\in  \mathbb{R}^{n-k}\times \mathbb{R}^k$ define, $g_{n,k,\alpha}$ as
   \begin{equation}\label{main_obstacle}
       g_{n,k,\alpha}(x,y) = \left\{
	\begin{array}{ll}
		g_{n,\alpha}(x,y),  & \mbox{if } x=0 ;\\
          \infty, & \text{otherwise}.
	\end{array}
\right.
   \end{equation}

\subsection{Barriers} 
For $n \geq 1$, one of our main barriers is
\begin{equation}\label{OneDelta}
	W_n(x) := \int_0^{|x|} \left(1 + s^n\right)^{\frac{1}{n}}\,ds,
\end{equation}
which solves
\begin{equation}\label{FlatSuper}
	\det D^2W_n = 1 + |B_1|\delta_0
\end{equation}
in the Alexandrov sense. It also satisfies
\begin{equation}\label{GrowthRadial}
	W_n(x) - \frac{1}{2}|x|^2 = \begin{cases}
		O(|x|), \quad n = 1 ;\\
		O(|\log|x||), \quad n = 2 ;\\
		c(n) + O(|x|^{2-n}), \quad n \geq 3,
	\end{cases}
\end{equation}
for some constant $c(n) > 0$, and 
\begin{equation}\label{DistGrowth}
	W_n(x) \geq |x|
\end{equation}
for all $x\in\mathbb{R}^n$.

We also define a useful family of Pogorelov-type barriers constructed in \cite{M1}.
We denote points in $\mathbb{R}^n$ by $(x,\,y)$ with $x \in \mathbb{R}^{n-k}$ and $y \in \mathbb{R}^k$. For $n \geq 3$ and $1 \leq k < \frac{n}{2}$, define the function $w_{n,\,k}$ on $\mathbb{R}^n$ by
\begin{equation}\label{PogoEx}
	w_{k,\,n}(x,\,y) = C(n) |x|^{2-2k/n}(1+|y|^2).
\end{equation}
For $C(n)$ sufficiently large we have
$$\det D^2w_{k,\,n} \geq 1$$
 in the slab $\{|y| < \rho_n\}$ for some $\rho_n > 0$.

\section{Connection with Optimal Transport}\label{examples}
Let $X$ and $Y$ be two bounded open sets in $\mathbb{R}^n$ such that $|X|=|Y|=1$. Here $|A|$ denotes the Lebesgue measure of $A$ in $\mathbb{R}^n$. We work with quadratic cost optimal transport between $X$ and $Y$. First of all, an optimal transport map $T$ exists, and is gradient of a convex function $u$.  According to Brenier’s Theorem (see \cite{B1, B2}), there exists a convex function $u : \mathbb{R}^n \to \mathbb{R}$ such that $\nabla u_{\#} \chi_X = \chi_Y$ and $\nabla u(x) \in \overline{Y}$ for a.e.\ $x \in \mathbb{R}^n$. As shown by Caffarelli (see \cite{C1}), if $Y$ is convex, then $u$ satisfies $\det(D^2 u)=1$ in Alexandrov sense, that is, for any Borel measurable set $A\subset \mathbb{R}^n$, we have $|A\cap X|=|\partial u(A)|$.
 When $X$ and $Y$ are convex, the regularity theory is well understood (see \cite{C2}, \cite{SY}, and \cite{CT}).
 
 In this section we focus on examples when one of the domains is non-convex. Through these examples, we see that one can expect to see the types of singularities we are constructing. For regularity results of optimal transport map when one of the domains is nonconvex, see \cite{MR2}, and \cite{CJLPR}. \\

 In \cite{C1}, Caffarelli gave an example when $Y$ is not convex, the potential of the optimal transport map is not $C^1$. We recall this example.

Let $n=2$, and let $X=B_1\subset \mathbb{R}^2$ be the unit ball centered at the origin. Let
$$
Y := (B_1^+ + e_1)\cup (B_1^- - e_1),
$$
where
$$
B_1^+ := B_1\cap \{x>0\}, \qquad B_1^- := B_1\cap \{x<0\},
$$
and $\{e_1,e_2\}$ denotes the standard basis of $\mathbb{R}^2$. The optimal transport map is given by
$$
T(x,y):=
\begin{cases}
(x+1,y) ,& \text{if } \{x>0\}\cap X;\\
(x-1,y), & \text{if } \{x<0\}\cap X,
\end{cases}
$$
which coincides with the gradient of the convex function
$$
u(x,y)=\frac{|x|^2+|y|^2}{2}+|x|.
$$
It is easy to check that $\det(D^2u)=1+2\mathcal{H}^1\vline_{\{x=0\}}$ in Alexandrov sense, that is, singularities arise along $y$ axis. For a detailed study of the optimal transport map in Caffarelli's example when the target domain remains connected, see \cite{CJLPR}.\\

Motivated by Caffarelli's example, we construct a series of new examples of optimal transport map in $\mathbb{R}^2$ where we get singularities supported on different $1$ dimensional sets. We also use the back and forth method developed in \cite{JL} to understand the singular structures in the following examples.

%% Add Caffarelli example and it's generalizations as well!!

\begin{ex}[\textbf{Framed Diamond}]\label{main_example}
\begin{figure}[ht]
    \centering
    \includegraphics[width=0.7\textwidth]{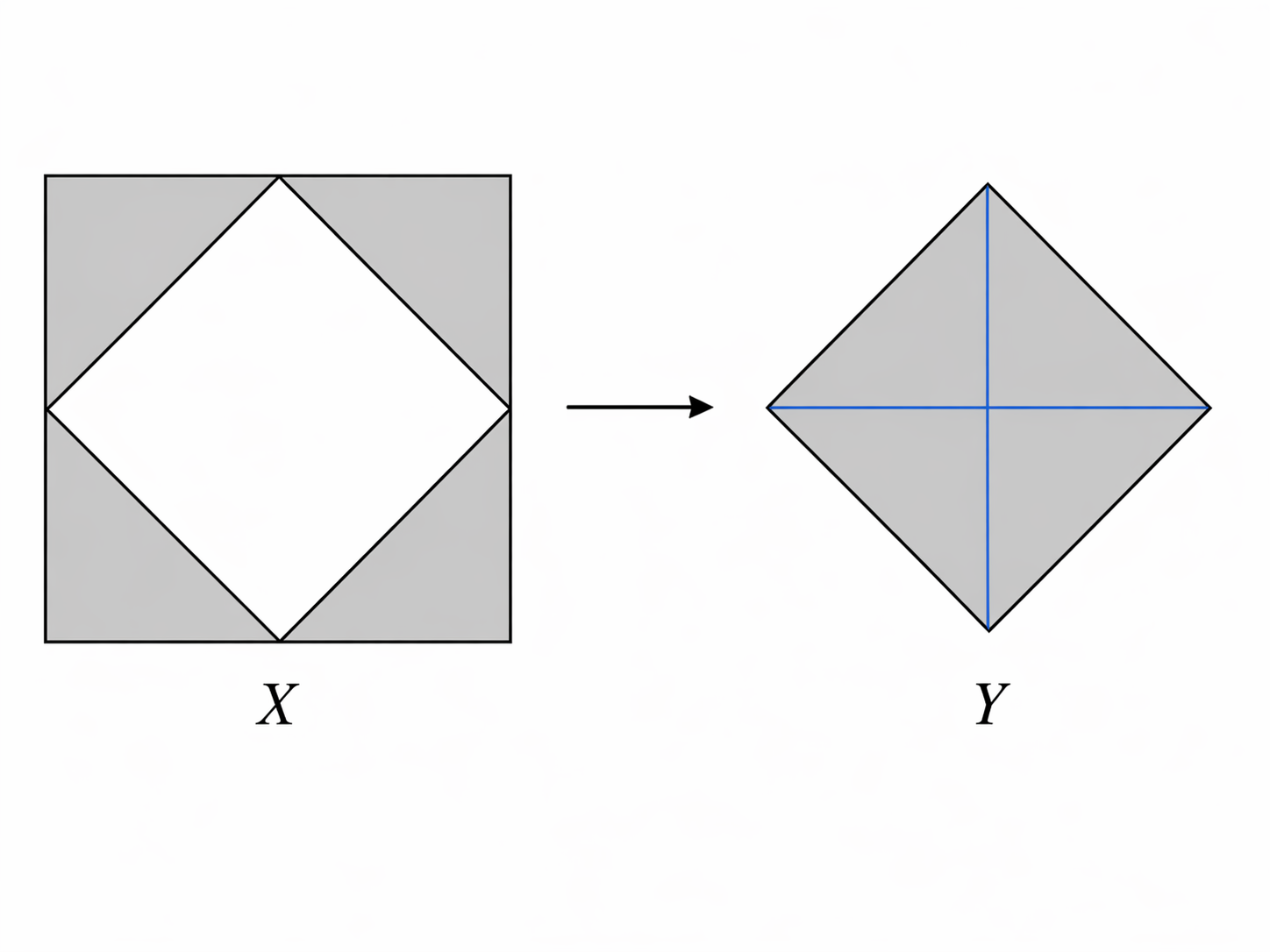}
    \caption{Optimal Transport with a `plus' singularity }
    \label{fig:plus}
\end{figure}

\begin{figure}[h]
    \centering
    \includegraphics[width=1\textwidth,keepaspectratio]{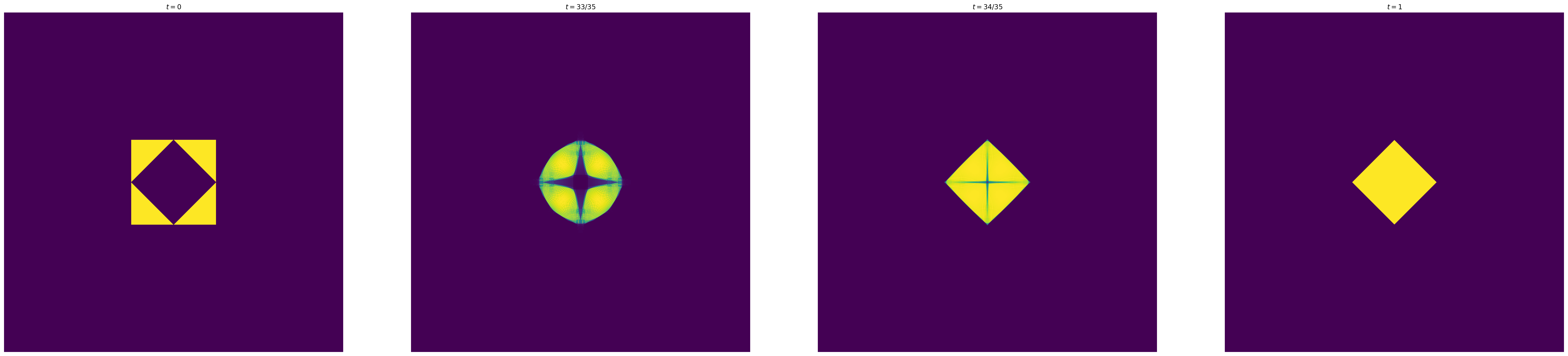}
    \caption{Displacement Interpolation between $X$ and $Y$}
    \label{fig:main_examplen}
\end{figure}

    We construct an example with a cross singularity arising in optimal transport in $\mathbb{R}^2$.
    Let $A = \{(x,y)\in\mathbb{R}^2: \max\{|x|,|y|\} < 1\}$ and $B = \{(x,y)\in\mathbb{R}^2:|x|+|y| < 1\}$. Consider the domains $X = A\setminus \bar{B}$ and $Y = B$ (see Figure \ref{fig:plus}). Let $T: X\to Y$ be  the optimal transport map from $X$ to $Y$. We show that the dual optimal transport map $T^{*}: Y \to X$ has a cross singularity in the sense that the potential of the optimal transport map is not differentiable (blue line in Figure \ref{fig:plus}).\\
    By symmetry and cyclic monotonicity, one can argue that $T$ maps each quadrant to itself. Let $Q_1 = \{(x,y)\in\mathbb{R}^2\;|\; x> 0,y > 0\}$. Consider the map $T_1: X\cap Q_1 \to Y\cap Q_1$.  By Remark 2.2 in \cite{AC}, $T_1^*$ is continuous  in the interior. This gives that there is no singular point for $T^*$ except maybe along $xy=0$.

    %Now that we have restricted $T$ to a mapping between two symmetric convex domains, one can expect boundary points mapping to boundary points( note that it is not immediate since our boundaries are not uniformly convex). Indeed, since our domains are reflections about the line $x+y=1$, we have $T_1^{*} = T_1\circ R$. Here $R$ is the reflection about the line $x+y=1$. Thus if some boundary point $x\in \partial (X\cap Q_1)$ gets mapped to some interior point $y\in Y\cap Q_1$, one would have $Ry\in X\cap Q_1$ maps to the boundary. However, since $Ry$ is in the interior, the map should be a smooth diffeomorphism in an open neighborhood of it. However, there cannot be a diffeomorphism from an open set to a set with boundary. Thus symmetry and interior smoothness imply that boundary points must map to boundary points.\\
    Any point on $x$ axis in $Y$ has to get mapped to two points on $X$ by symmetry. Thus, every point on $x$ axis (similarly on $y$ axis) is a singular point for $T^*$. Therefore the potential of the dual optimal transport plan satisfies an equation of the form, $\det(D^2u) = 1+ f\mathcal{H}^1(\Sigma)+a_0\delta_0$. Here $\Sigma$ is the blue line in Figure \ref{fig:plus}, $f \in L^{\infty}(\Sigma)$ is a nonnegative function, and $a_0>0$ is a fixed constant.
    
    %We now consider the set $Y_1 = (\bar{Y}\cap\{x=0\})\setminus \{T_1((0,1))\}$. One can see that $T_1^{*}(Y_1) \subseteq Q_1$. However, if we restrict the domain to the second quadrant, $Y_1$ remains part of the boundary of the target domain. Conclude  $T^{*}:Y\to X$ splits mass along $Y_1$, that is,  we have singularities along $Y_1$. By symmetry it will result in a cross shaped singularity. In Figure \ref{plus}, the blue line represents the singularities of $T^*$.
    %add blue line is the singularity

\end{ex}

\begin{rem}
The above construction can be generalized by considering an $n$-gon together with a suitably rotated and rescaled copy inside it. This yields examples in which several singular line segments intersect at a single point, forming a star-shaped configuration. Similar singularities can also be generated by modifying Caffarelli's example if we change the target domain by spreading it out into $k$ equal parts, instead of splitting it into two halves.
\end{rem}

\begin{ex}[\textbf{Square Frame}]\label{square}

%add the triangle example etc
Let $Q = \{\max\{|x|,|y|\} < 1\} \subset \mathbb{R}^2$ be a square. Set $\lambda =0.8$ and define
$$
\lambda Q = \{ \lambda x : x \in Q \}.
$$
Consider the optimal transport map taking uniform density on $Q\setminus \lambda Q \:\:\text{to} \:\:\lambda Q$ and denote it by 
$
\mathcal{T} : Q \setminus \lambda Q \to \lambda Q
$
such that
$
\mathcal{T}_{\#} \mu = \nu,$ where $\mu$ and $\nu$ are  uniform source and target density. Then as per the numerical models as in \cite{JL}, the singularities of the dual optimal transport map look like a cross singularity (see Figure \ref{fig:numerical}). Figure \ref{fig:numerical} represents displacement interpolation, as defined in \cite{McCann}, between the domains $Q\setminus \lambda Q$ (on the left) and $\lambda Q$ (on the right).
\begin{figure}[h]
    \centering
    \includegraphics[width=1\textwidth,keepaspectratio]{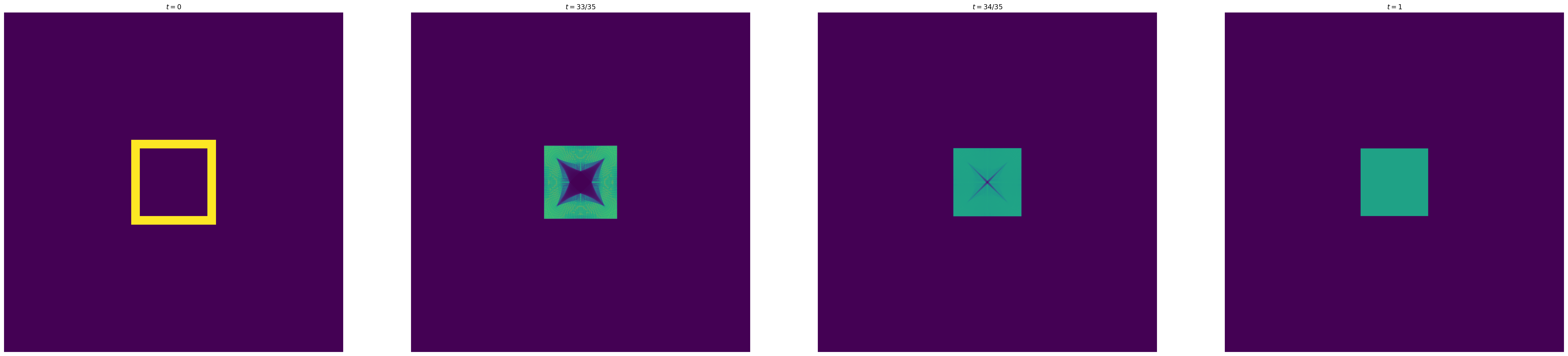}
    \caption{Displacement Interpolation between $Q\setminus \lambda Q$ and $Q$.}
    \label{fig:numerical}
\end{figure}

\noindent
By exploiting symmetry of the source and target domain as in example \ref{main_example} one can see that the singular set will be a subset of the diagonals of the square $\lambda Q$. The numerical model shows it is in fact a cross. Therefore the potential of the optimal transport plan satisfies an equation of the form, $\det(D^2u) = 1+ f\mathcal{H}^1(\Sigma)+a_0\delta_0$. Here $\Sigma$ represents the diagonals of the target domain, $f\in L^{\infty}(\Sigma)$ is a nonnegative function, and $a_0>0$ is a constant.

\end{ex}

\begin{ex}[\textbf{Pacman}] Define 
$$C = \{(x,y)\in\mathbb{R}^2:|x|^2+|y|^2 < 1\}$$ 
and
$$\Gamma = \{(x,y)\in\mathbb{R}^2: y < m|x|\}$$
for $m>0$ (for the numerical model, in Figure \ref{fig:pacman}, we have taken $m=2$). Consider the optimal transport map taking uniform densities on $C\setminus \Gamma$ to $C$ and denote it with $T:C\setminus \Gamma\to C$. By using numerical model as in \cite{JL}, the singularity of the dual optimal transport map is
a segment along $y$-axis starting from an interior point of $\Gamma$ and extending up to the boundary of $C$.
%(\textbf{maybe better way to say this:}) ray starting from a point on the $y-axis$ and going till the boundary of $C$ along the $y-$axis.
\begin{figure}[h]
    \centering
    \includegraphics[width=1\textwidth,keepaspectratio]{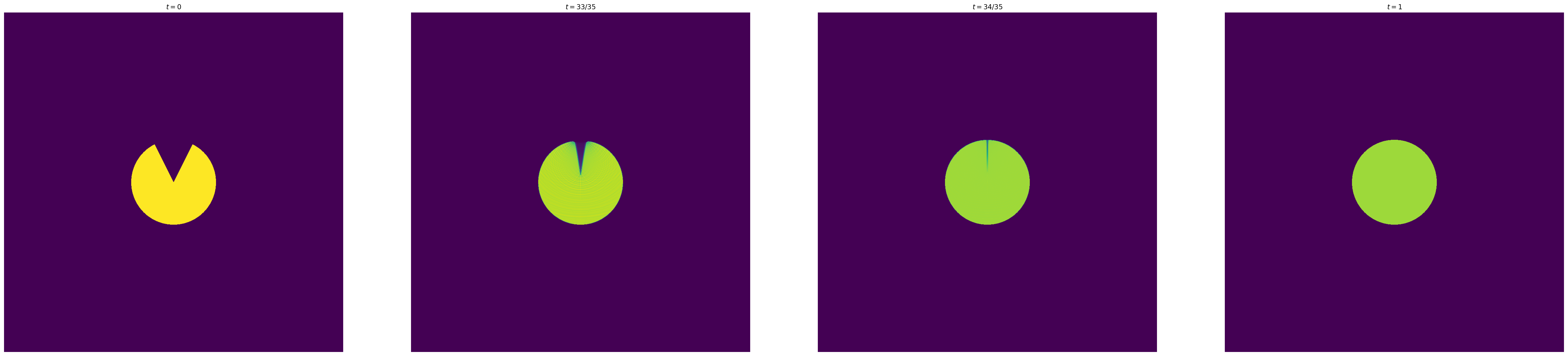}
    \caption{Displacement Interpolation between $C\setminus \Gamma$ and $C$.}
    \label{fig:pacman}
\end{figure}
Again, using symmetry and \cite[Remark 2.2]{AC}, one can see that the singular set has to be a subset of positive $y$ axis. The numerical model supports this.

\begin{ex}[\textbf{Cat's eye}] Define the domains 
$$C = \{(x,y)\in\mathbb{R}^2:|x|^2+|y|^2 < 1\}$$
and 
$$\mathcal{E}_e = \{(x,y)\in\mathbb{R}^2:|x|^2+\frac{|y|^2}{e} < \frac{2}{15}\}$$
for $e>1$ (in Figure \ref{fig:ellipse_numerical}, we have taken $e=5$). Consider optimal transport map taking uniform densities on $C\setminus \mathcal{E}$ to $C$ and denote it with $T_e:C\setminus \mathcal{E}_e\to C$. By using numerical model as in \cite{JL}, the singularity of the dual optimal transport map is along the $y-$axis and is strictly contained inside the domain $C$. In fact, one expects the potential of the optimal transport plan satisfies $\det(D^2u) = 1+f\mathcal{H}^1(L)$, where $L = \{(0,y)\;|\; |y| \leq \epsilon\}$ and $f\in L^{\infty}(L)$ is a nonnegative function. 
\begin{figure}[h]
    \centering
    \includegraphics[width=1\textwidth,keepaspectratio]{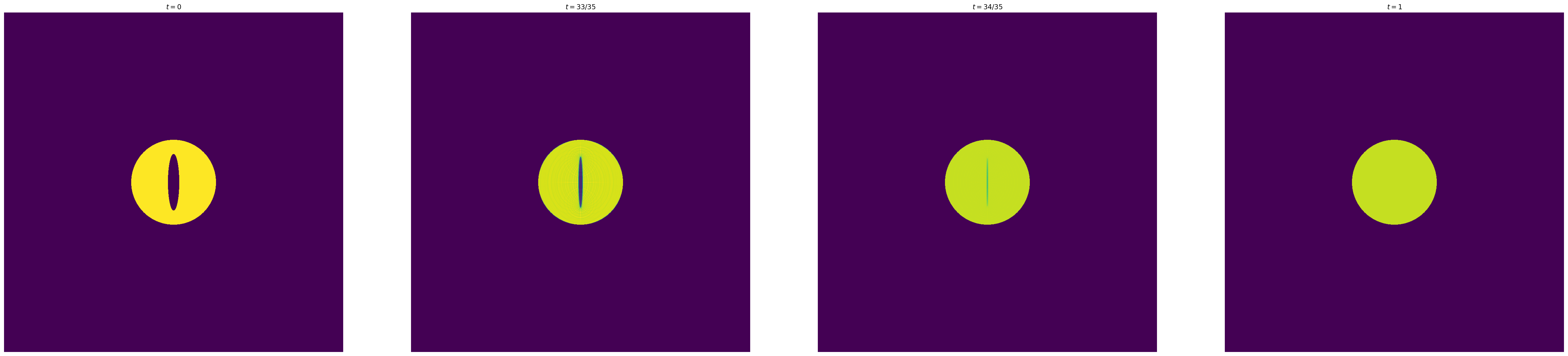}
    \caption{Displacement Interpolation between $C\setminus \mathcal{E}_e$ and $C$.}
    \label{fig:ellipse_numerical}
\end{figure}

%Heuristic: Note that in this example it is not possible to the study singularity by reducing to the case of mapping between convex domains. However as shown in \cite{MR2} $T_e \in C^{1/2}(C\setminus \mathcal{E}_e)$. Moreover by symmetry we can still restrict to mappings in a particular quadrant. Therefore for the dual optimal transport map the singularities must lie along $xy=0$. Now heuristically we expect it lie only on the $y-$axis is because if we consider the maps $T_e$ as $e\to \infty$, $\mathcal{E}_e$ converges to a subset of the $y-$axis. Moreover as $e\to \infty$, the point $(\frac{2}{15e},0)\in \partial \mathcal{E}_e$, is close enough to the origin that it can collapse to the origin instead of some other part of $\partial (C\setminus \mathcal{E}_e)$ mapping to the origin.

\end{ex}

\begin{rem}
    Taking the partial Legendre transform of solutions to Monge-Ampere equation, 
    \begin{equation}\label{sig_rem}
        \det(D^2u) = 1+f\mathcal{H}^1\vline_L,
    \end{equation}
    in the $x-$variable. Here $L$ is a line, 
    like in our Cat's eye example, satisfy an equation of the form,
    $$\Delta u^{*} = 0, \text{ except on } L $$
    and on $L$, $\partial_x u^*$ has a jump discontinuity. This shows a connection between Optimal transport of non-convex domains and the Signorini problem. Using standard theory from the thin obstacle problem (\cite{real}), this also shows that at the end points of $L$, $u^*$ separates from the thin obstacle at a rate of $r^{3/2}\cos(\frac{3}{2}\theta)$ in polar coordinates centered at one of the end points. This connection is not surprising given our construction of solutions to equations like \eqref{sig_rem} via an obstacle problem. We intend to investigate this connection further in a future work. 
\end{rem}

\noindent
%Heuristic: By exploiting symmetry and interior regularity as in example \ref{main_example}, we can conclude that singularity of the dual optimal transport map has to lie on the $y-$axis. Moreover since the origin $(0,0)$ is further away from the boundary point $(0,1)\in \partial C$ than $(\frac{1}{\sqrt{5}},\frac{2}{\sqrt{5}})\in \partial (C\setminus \Gamma) $, one expects a part of the curve $y=m|x|$ to map to the $y-$axis. This would result in a singularity for the dual optimal transport map.
\end{ex}

\section{Main proof}\label{proof}
In this section we prove the theorems stated in Section \ref{intro}. Recall that we write our coordinates as $(x,y)\in \mathbb{R}^{n-k}\times \mathbb{R}^k$. For $a>0$, define $S_a =\{0\}^{n-k}\times (-a,a)^k$. For two constants $p,q$, an assumption of the form $p\ll q$ means there exists a constant $c$  such that when $p\leq cq$ then conclusion holds, where $c$ depends on $n$.

%\begin{thm}\label{line}
%    For any $a>0$ and $\alpha\in (0,1)$,
%    there exists a convex function $u\in C^\infty(\mathbb{R}^n\setminus S_a)$, and a non-negative function $f \in L^{\infty}(S_a)$, that solves
 %   $$\det(D^2u) = 1+f(0,y)\mathcal{H}^{k}|_{S_a}$$
%    such that $f \geq c>0$ on $S_{\alpha a}$, where $c$ depends on $a$ and $n$ .\end{thm}
\subsection{Proof of Theorem \ref{line}}
\begin{proof} 
    For $\epsilon>0$ small, to be fixed later, we quadratically rescale everything so that $a = \epsilon$. 
     We now solve the obstacle problem as in Lemma \ref{obstacle_problem} with parameters $g(x)=g_{n,k,\alpha}(\epsilon x)$, $\phi = W_n+10$, $\mu = 1 dx$ and $U = B_R$ for any $R>1$. For these choices of parameters and $\mathcal{F}$ as in \eqref{class}, define $u_R:=\sup_{v\in \mathcal{F}}v $. \\

\textit{Boundedness from Above:} Since $u_R\in\mathcal{F}$,  we have $\det(D^2(W_n+10))=1\leq \det(D^2u_R)$   on $B_R\setminus \{0\}$, and $u_R\leq W_n+10$ on $\partial B_R\cup \{0\}$. Thus by maximum principle $u_R \leq W_n+10$ on $B_R$.\\
%You do not need +10 here.
%\textbf{Discuss changes here!}

\textit{Persistence of Singularity for all $R$:} Consider $D(x) = W_n(x)-W_n(\epsilon \textbf{e})$ where $\textbf{e}\in \mathbb{S}^{n-1}$. Since $D \leq 0 \leq g$ in $B_\epsilon$,  we have $D\in \mathcal{F}$. Choose $\rho$ such that $\epsilon \ll\rho \ll 1$. On $\partial B_\rho$ we have
$$D = W_n(\rho \textbf{e})-W_n(\epsilon \textbf{e}) \geq \rho - W_n(\epsilon \textbf{e}) \geq \rho/2.$$
Since $\rho\ll 1$, we have $\rho^2/2 \leq \rho/4$. Consider,
$$\Phi = \frac{|(x,y)|^2}{2}+\frac{|x|}{4}.$$
    We have $\Phi \leq g$ and $\Phi \leq D$ on $\partial B_\rho$. Define $\tilde{\Phi}$ to be $ \max\{\Phi,D\}$ in $B_\rho$ and $D$ in the compliment of $B_\rho$. Since $\tilde{\Phi}\in\mathcal{F}$, we have $u_R\geq \tilde{\Phi}$ and it satisfies $\det(D^2u_R) \geq 1+\tilde{f}_R\mathcal{H}^{k}|_{S_{\epsilon}}$, for some $\tilde{f}_R\geq 1/2$ on $S_{\alpha \epsilon}$. Note that, by construction, on the set $S_{\epsilon}$ we have that $\{u_R<g\}$ is non-empty, otherwise $u_R$ would have unbounded gradient. We also note that on the relative boundary of $ \{u_R<g\}$ inside $S_{\epsilon}$, $u_R$ cannot develop a singular measure other than what we are trying to construct. If not, let $x$ be such a point. Note that $u_R(x)=g(x)$ and $u_R \leq g$. For $\textbf{v}\in \text{span}\{e_{n-k+1},\cdots,e_{n}\}\cap \mathbb{S}^{n-1}$, if $p\in \partial u(x)$,  then $\langle p,\textbf{v}\rangle = \langle\nabla g_{n,\alpha},\textbf{v}\rangle$. This shows $p$ lies on a subset of a $n-k$ dimensional plane in $\mathbb{R}^n$. Therefore, $u_R$ solves $\det(D^2u_R) = 1+f_R\mathcal{H}^{k}|_{S_{\epsilon}}$ for some $f_R\geq 1/2$ on $S_{\alpha\epsilon}$. \\

\textit{Regularity:} We know $D\leq u_R\leq W_n+10$ for all $R>1$. For a uniformly bounded sequence of convex function, we have a subsequential convergence by Arzelà–Ascoli theorem. Up to a subsequence, define this limit by  $u:=\lim_{R\to\infty}u_R$. 
By stability of  solutions to Monge–Ampère equation under convergence, we have
$\mathcal{M}u_R \to \mathcal{M}u$ weakly as measures. Furthermore,
$\mathcal{M}u=1$ outside of $S_{\epsilon}$. Thus by taking open convex neighborhoods of $S_{\epsilon}$, converging to $S_{\epsilon}$, we get that $u\in C^\infty(\mathbb{R}^n\setminus S_{\epsilon})$ (see Theorem \ref{CL3_exterior}).
\end{proof}
\begin{rem}
   One could start with a simpler function on $\mathbb{R}$, such as $|x|^2/2$, instead of $g_1(x)$ defined in Section \ref{priliminaries} to construct the obstacle. The main difficulty would then be to avoid a singularity at the relative boundary of ${u_R < g}$ within $S_{\epsilon}$.
\end{rem}

\subsection{Proof of Theorem \ref{polytope}}
Recall that $P$ is a convex polytope such that its vertices lie on or outside the sphere $B_a$. For $0<\alpha<1$, we define $\Lambda_\alpha := \partial P \cap B_{\alpha a}$, where $\partial P$ denotes the boundary of $P$ in $\mathbb{R}^n$. Also recall that $P[k]$ denotes the $k$-skeleton of $P$. We use the notation $\omega(h)$ to denote a function with $\lim_{h\to 0} \frac{h}{\omega(h)} = 0$.

%\begin{thm}\label{polytope}
%    For $k = 1$ or $2$ and $0<\alpha<1$, there exists a convex function $u\in C^\infty(\mathbb{R}^n\setminus P[n-k])$ that solves 
%    $$\det(D^2u) = 1+f\mathcal{H}^{n-k}|_{P[n-k]},$$
%     where $f\in L^{\infty}(\partial P)$, $f \geq c>0$ on $ P[n-k]$ and $c$ depends on $a$ and $n$.
%\end{thm}
\begin{proof}
First, we perform a quadratic rescaling so that $a = \epsilon$.  For $\epsilon>0$ small  to be fixed later, we solve the obstacle problem as in Lemma \ref{obstacle_problem} with parameters $g(x)=g_{n,\alpha}(\epsilon x)|_{P[n-k]}$, $\phi = W_n+10$, $\mu = 1 dx$ and $U = B_R$ for any $R>1$. For these choices of parameters and $\mathcal{F}$ as in \eqref{class}, define $u_R:=\sup_{v\in \mathcal{F}}v$. \\
    
    \textit{Boundedness from Above:} Boundedness from above follows similarly as in the proof of Theorem \ref{line}.\\
%add a line recalling our coordinate (x,y)
    \textit{Persistence of Singularity for all $R$:} We can construct a linear function $L$ that is zero along an $(n-k)$ dimensional face, has $|\nabla L|=1/8$ and $P\subseteq \{L<0\}$. We set up our problem as in Theorem \ref{line} by considering $D = W_n-W_n(\epsilon \textbf{e})$ where $\textbf{e}\in\mathbb{S}^{n-1}$. We have $D\in \mathcal{F}$ since $D \leq 0 \leq g$ in $B_\epsilon$. Consider $\rho$ such that $\epsilon \ll\rho \ll 1$. On $\partial B_\rho$ we have
$$D = W_n(\rho \textbf{e})-W_n(\epsilon \textbf{e}) \geq \rho - W_n(\epsilon \textbf{e}) \geq \rho/2.$$
By our choice of $\rho$ we can ensure $\rho^2/2 \ll \rho/16$. By a rotation of coordinates, we can choose one of the faces, say $F$, to be along $|x|=0$. Now by considering $L$ for this face and by another possible rotation, one can ensure $P\setminus F\subseteq \{x_1>0,x_2>0,\cdots,x_{n-k}>0\}$. Consider
$$\Phi = \frac{|(x,y)|^2}{2}+\frac{|x|}{16}+L.$$
Clearly $\Phi$ touches the obstacle $g$ along the face $F$ in $\Lambda_{\alpha}$. Moreover by construction, $\frac{|x|}{16}+L = -\frac{|x|}{16}\leq 0$ on $P$, we get that $\Phi\leq \frac{|(x,y)|^2}{2}\leq g$ on $P$. The value of $\Phi$ on $\partial B_\rho$ can also be bounded by $\rho/4$. Thus, by defining $\tilde{\Phi} = \max\{\Phi,D\}$ in $B_\rho$ and $D$ outside, we can now see $u_R$ satisfies $\det(D^2u_R)\geq 1+f_R\mathcal{H}^{n-k}|_{P[n-k]}$, where $f_R \geq 1/8$ on $\Lambda_\alpha$. Again since our obstacle is smooth on $\partial P$, by arguing similarly to Theorem \ref{line} we get that the singular part of the measure is only of the form we are trying to construct.
 \\

    \textit{Regularity:} By boundedness from above and below as in Theorem \ref{line},  we can define $u:=\lim_{R\to \infty} u_R$. Moreover, we get that $\mathcal{M}u_R \to \mathcal{M}u$. Furthermore we can see $\mathcal{M}u=1$ outside of $P$. Thus, by \cite{CL3} $u\in C^\infty(\mathbb{R}^n\setminus P)$. Now it is possible that singularities arise in $(P[n-k])^c\cap P$. However, since we have  $\det(D^2u)=1$ on $(P[n-k])^c$, any singular set $\{u=L\}$ must propagate at least up to $ P[n-k]$ . Let $p\in  P[n-k]\cap \{u=L\}$. Note that if $u(p)< g(p)$, then this propagation cannot terminate here and will propagate to infinity (see Proposition \ref{no_extremal}). However we know smoothness of $u$ outside $P$. Thus, we get $u(p)=g(p)$. Note that $u\leq g$ on $ P[n-k]$. Recall we denote a section at $p$ of height $h$ by $S_h(p)$. By smoothness of $g$ along the boundary, $B_{\delta\sqrt{h}}\cap  P[n-k] \subseteq S_h(p)\cap \partial P$ for some non-zero constant $\delta$ and $h$ small. We also know that $S_h(p)$ contains at least a line segment that has some non-zero length $l$. Thus, we get that $(B_{\delta\sqrt{h}}\cap  P[n-k] )\cup \{u=L\} \subseteq S_h(p) $. Since the volume of the convex hull of $(B_{\delta\sqrt{h}}\cap  P[n-k] )\cup \{u=L\}$ is at least $c_0 (\delta\sqrt{h})^{n-k}l\cdot \omega(h)^{k-1}$, the volume of $S_h(p)$ is at least $\omega(h^{(n+k-2)/2})$. Since $u$ solves $\det(D^2u) \geq 1$, the volume of a section can at most be $h^{n/2}$ (see \cite[Lemma 2.2, Lemma 2.3]{M2}) up to a constant. This gives us a contradiction.
    %Cite Caffarelli's original paper
\end{proof}
% \begin{rem}\label{n-2rem}
%     Note that in the regularity proof above we could do better and reduce the dimension of the singularity in the above theorem to $n-2$. This is because then in $n-2$ dimensions we have quadratic growth of a section, in at least one dimension we have a line segment and in another direction we can atmost be lipschitz. Infact in the leftover dimension, by subtracting an appropriate linear function our growth can atmost be $o(h)$, thus the volume of a section is at least $\omega(h^{(n-2)/2}h) = \omega(h^{n/2})$. Which also leads to a contradiction (See \cite[Lemma 2.3]{M2} ).
% \end{rem}

\begin{rem}\label{partial_data}
    \begin{figure}[ht]
    \centering
    \includegraphics[width=0.4\textwidth]{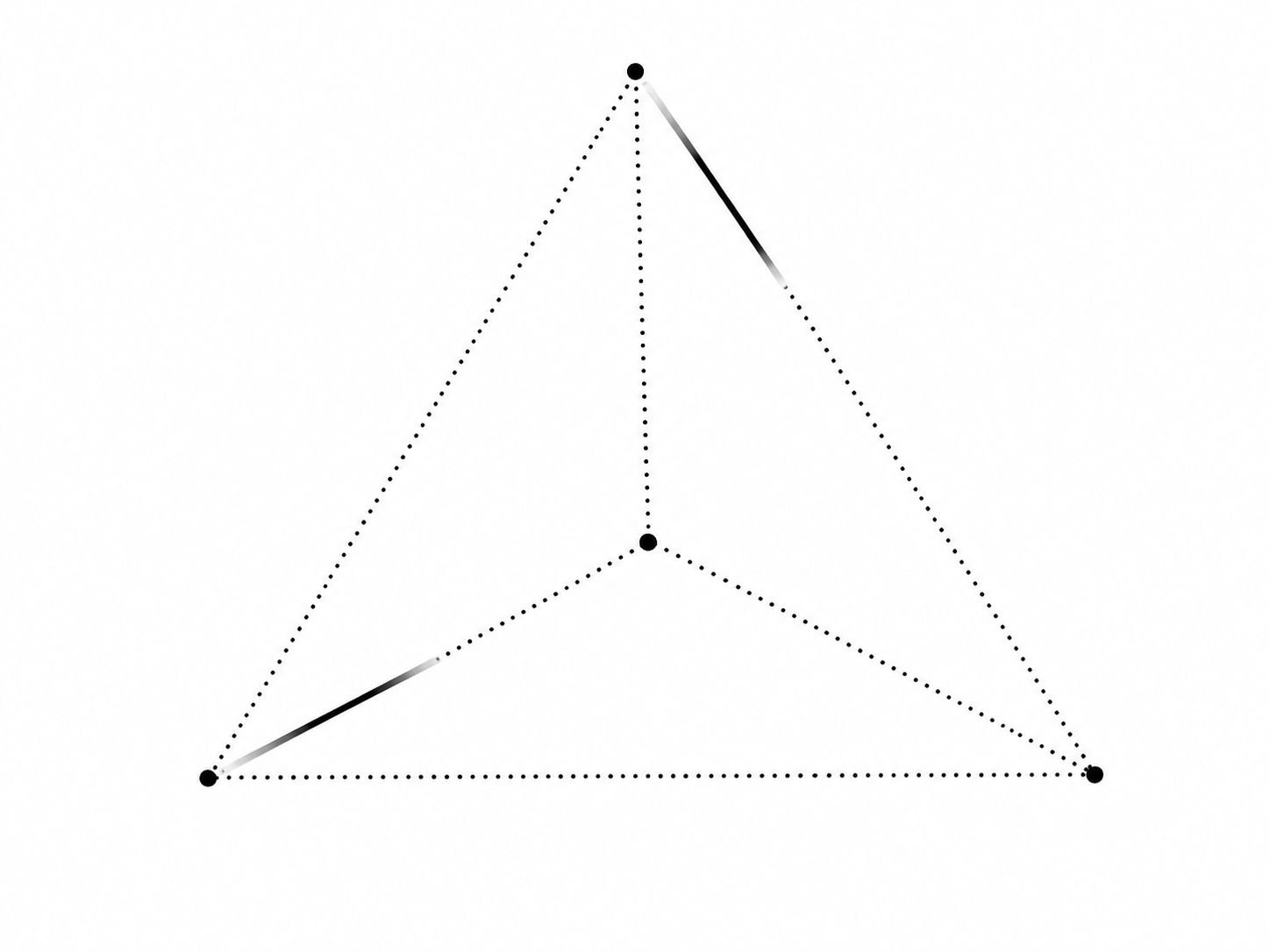}
    \caption{Partial prescription of obstacle on a tetrahedron}
    \label{partial}
\end{figure}
    By examining the proof above, one can consider any convex polytope $P$, and prescribe data on only a portion of the boundary and the result still holds. In particular, our results also cover situations where lower dimensional singularities might not be co-planar and/or connected (see Figure \ref{partial}). This is in stark contrast to singularities constructed in \cite{M1} and \cite{MR1} where this configuration cannot occur.

\end{rem}

\begin{rem}\label{interaction} We analyze dimensions $n=2,3$ in particular. In dimensions $n=2,3$, our theorem above can yield one-dimensional singularities with a solution, $u$, solving $\det(D^2u)=1$ away from those singularities. When comparing to \cite{M1}, and \cite{MR1}, we can see we have a milder singular Monge-Ampère measure in the sense that the singularity is distributed along a line segment rather than at a few points.

\end{rem}

\subsection{Proof of Theorem \ref{cross}}
%\begin{thm}\label{cross}
%    Let dimension $n=2,3$. Let $L_1,L_2,\cdots, L_k$ be line segments starting from the origin. We can construct a solution
%    $$\det(D^2u) = 1+c_0\delta_0+f\sum_{i=1}^k\mathcal{H}^1|_{L_i}$$ 
%    such that $u\in C^{\infty}(\mathbb{R}^n\setminus \{L_1,\cdots,L_k\})$ and $f\geq c_1$ on $\cup_{i=1}^k \alpha L_i$, for any $0<\alpha<1$ and for some positive constants $c_0,c_1$ depending on lengths of line segments and dimension.

%\end{thm}
\begin{proof}
    We quadratically rescale everything so that all the line segments lie in $B_\epsilon$, where $\epsilon>0$ is a constant to be fixed later. Also, let $d_i$ denote the direction of each line segment $L_i$ for each $i=1,\cdots,k$. We now solve the obstacle problem as in Lemma \ref{obstacle_problem} with parameters $g(x)=(g_{n,\alpha}+\sum_i \text{dist}(\cdot, \langle d_i\rangle)/8k)(\epsilon x)\big|_{\cup L_i}$, $\phi = W_n+10$, $\mu = 1dx$ and $U = B_R$ for any $R>1$. For any vector $e\in \mathbb{R}^n$, we denote by $\langle e\rangle$ the linear subspace generated by vector $e$. We denote the solution to be $u_R$.\\
        
    \textit{Boundedness from Above:} Boundedness from above follows similarly as in the
    proof of Theorem \ref{line}.\\
    
    \textit{Persistence of singularities:} We set up the problem as before by considering $D = W_n-W_n(\epsilon \textbf{e})$ for $\textbf{e}\in \mathbb{S}^{n-1}$. We have $D\in \mathcal{F}$ since $D \leq 0 \leq g$ in $B_\epsilon$. Now choose $\epsilon$ and $\rho$ such that $\epsilon \ll\rho \ll 1$. On $\partial B_\rho$ we have 
$$D = W_n(\rho\textbf{e})-W_n(\epsilon\textbf{e}) \geq \rho - W_n(\epsilon\textbf{e}) \geq \rho/2.$$
By our choice of $\rho$, we can ensure $\rho^2/2 \leq \rho/16$.  Consider the barrier,
$$\Phi = \frac{|x|^2}{2}+\sum_{i=1}^k\frac{\text{dist}(x,\langle d_i\rangle)}{8k}.$$
It is easy to see $\Phi \leq g$. Moreover, on $\partial B_\rho$, $\Phi \leq \rho^2/2+\rho/8<\rho/2$. We also note that when restricted to $L_i$, $\Phi$ is a smooth function. Thus with arguments as in Theorem \ref{line}, \ref{polytope}, we have $\det(D^2u_R) = 1+f_R\mathcal{H}^1|_{\cup L_i}$. 
\\

    \textit{Regularity:} Regularity follows similarly to the proof in Theorem \ref{polytope}.
\end{proof}

\begin{rem}
    For some $M>1$, if one takes $M|x|^2$ as an obstacle in the proof of Theorem \ref{cross}, one can not obtain a $Y$ shaped singularity in general, since at the origin, one won't be able to develop a singular mass. However, one would still get some $\mathcal{H}^1$ type singularities along each of the line segments by following the argument above with an appropriately chosen weight $w_i$ in the barrier $\frac{|x-\cdot|^2}{2}+w_i\text{dist}(x,\langle e_i\rangle)+L_0$. The choice of weights $w_i$ depends on the relative arrangement of the lines $L_i$ and the point we are at. Here $L_0$ is the supporting hyperplane to $M|x|^2$ at some point on $L_i$.  
\end{rem}

%By examining the proof of \ref{polytope} one can prove the following theorem,

\subsection{Proof of Theorem  \ref{smooth_boundary}}

\begin{proof}
    The proof follows by simply using $|x|^2/2$ (or any smooth convex function) restricted to $\partial \Omega$ as an obstacle and using the techniques developed above.
\end{proof}
 \begin{figure}[ht]
    \centering
    \includegraphics[width=0.4\textwidth]{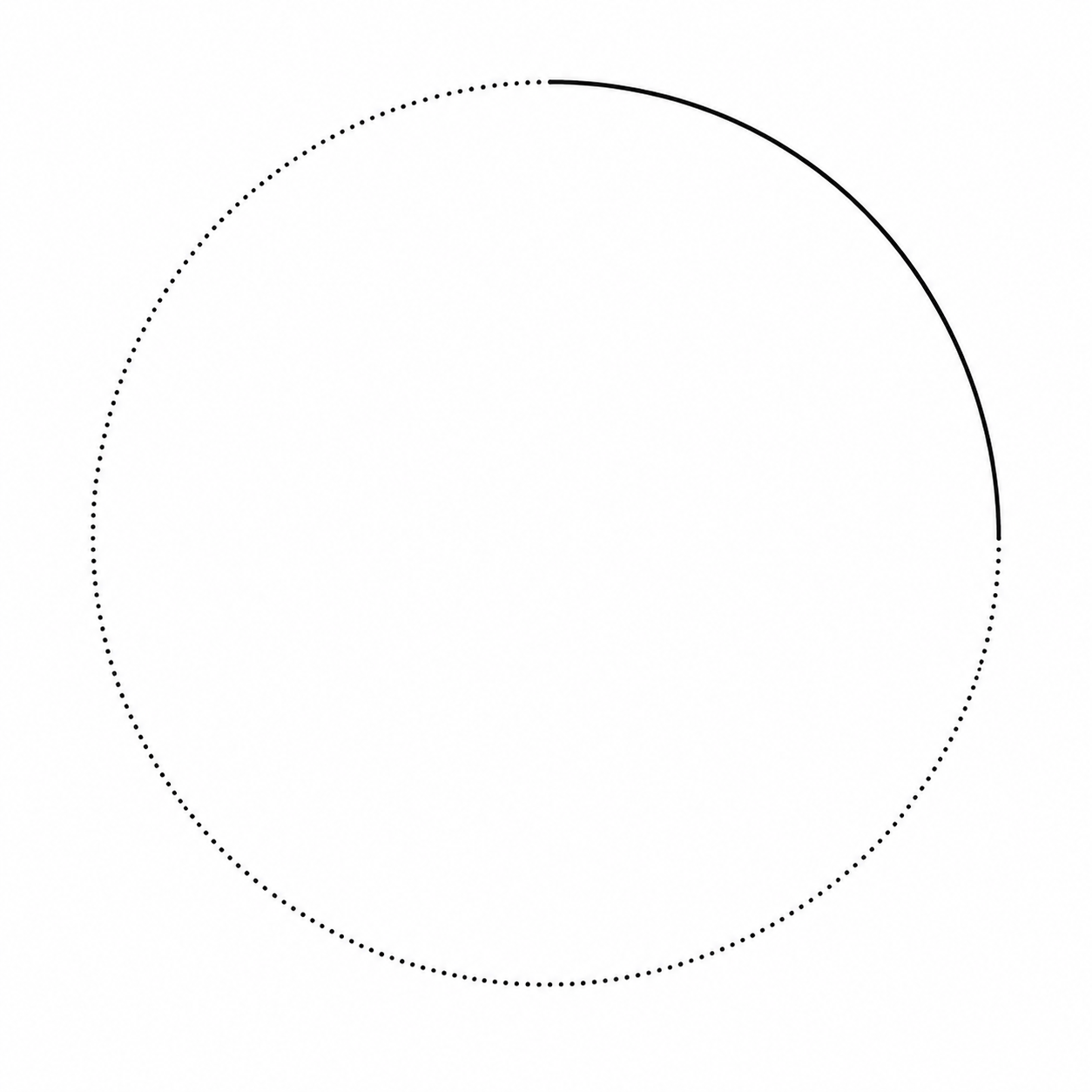}
    \caption{Partial prescription of obstacle on a circle}
    \label{circ}
\end{figure}

\begin{rem}

   If we want to partially prescribe $\mathcal{H}^{n-1}$ measure on parts of the boundary, then we need to work with obstacles built in Section \ref{obstacles} as discussed in Remark \ref{partial_data}. For example, in $\mathbb{R}^2$, we can partially prescribe the data only on a part  a circle and get a singularity like Figure \ref{circ}.
\end{rem}

%The above Theorem \ref{smooth_boundary} can be thought of as performing the Perron's method to find solution to the Dirichlet problem for Monge-Ampere equation in a "global" fashion.

\section{Interaction of Singularities}\label{interaction_section}

As discussed in the introduction, we take codimensions $k=1,2$ in the statement of Theorem \ref{polytope} to avoid the development of singularities of the solution away from the support of the obstacle. In this section we investigate the case when the codimension $k\geq 3$. We observe through some simple examples that singularities start to “communicate". We also note that this interaction depends strongly on the geometry and arrangement of the obstacles.\\

Consider the following examples:
\begin{ex}\label{integrate}
  In dimension 4, denote the coordinates as $(x,y,z)\in \mathbb{R}\times\mathbb{R}^2\times\mathbb{R}$. Now consider a function $\varphi = x^2/2+w_{1,3}(y,z)$. One can check that it solves $\det(D^2\varphi)\geq 1$ on the slab $|z| \leq \rho_n$, has singularities on $|y|=0$, and strict convexity fails along lines parallel to the $z-$axis. However in one perpendicular direction, namely along the $x-$axis, it has quadratic growth.
  
  Thus in the case of partial prescription of data (see Remark \ref{partial_data}) along 2 parallel line segments: $z = \pm \rho_n/100$ on the $xz-$plane for $|x| \leq 1$, the solution of the obstacle problem \ref{obstacle_problem} with appropriate parameters as in the proofs of Theorems \ref{line}, \ref{polytope} and \ref{cross} would have a $1+1=2$ dimensional singularity by using $\varphi$ as a barrier from below. We gain an extra dimension for the singular set as the $1$ dimensional singularities “integrate" along the obstacle (see Figure \ref{lineinteract}).
  This shows, even in the obstacle problems we consider singularities can arise away from the obstacle. From the perspective of optimal transport, this shows singularities can also be produced by interaction between masses distributed along line segments, rather than just points as in \cite{M1} and \cite{MR1}.

\end{ex}
\begin{figure}[ht]
    \centering
    \includegraphics[width=0.6\textwidth]{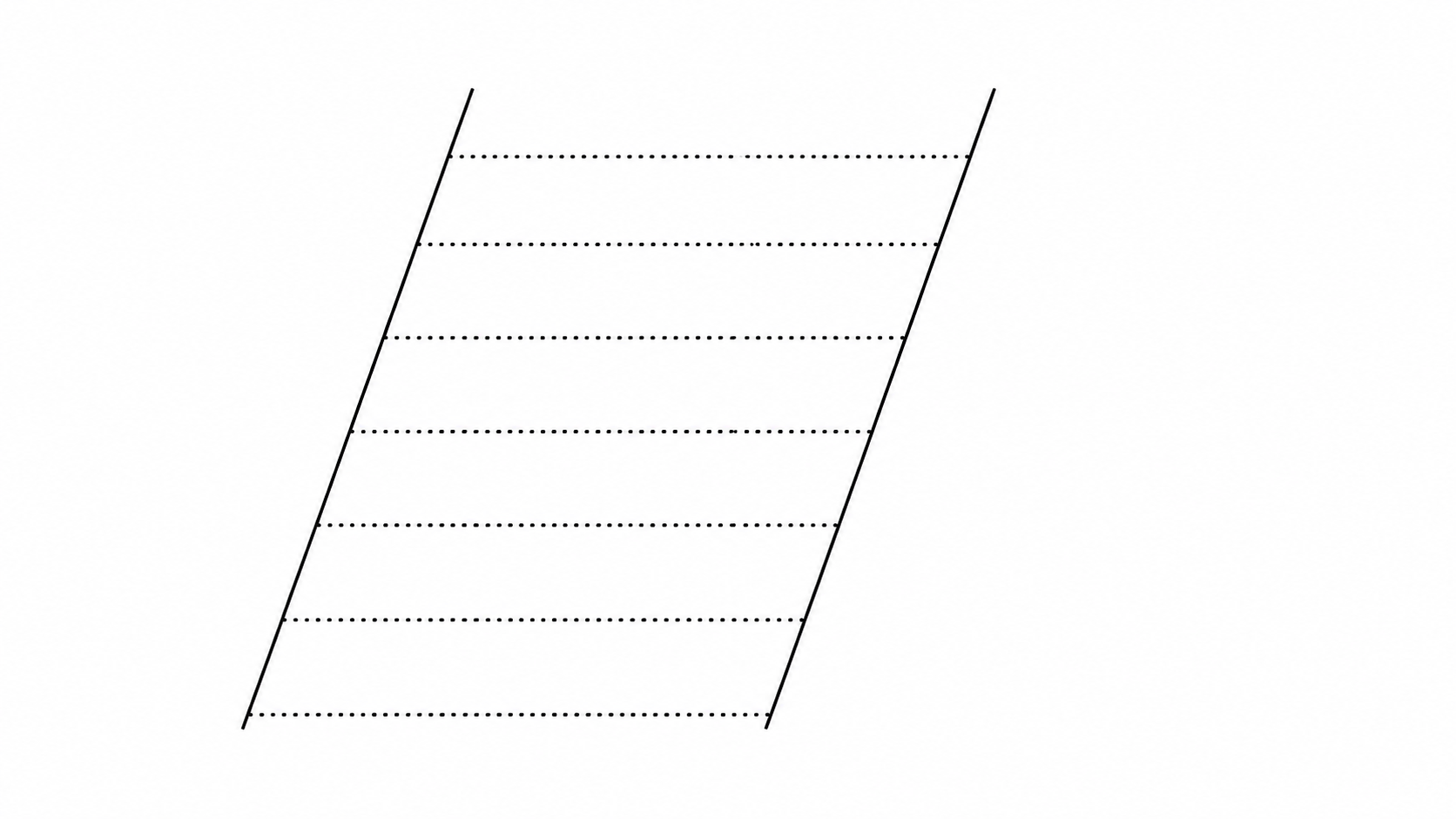}
    \caption{Interaction between line segments}
    \label{lineinteract}
\end{figure}

\begin{rem}
    In dimension 3, by placing the obstacle $v$ developed in \cite{M2} instead of the quadratic obstacle and using the relevant modified Pogorelov type barrier as in \cite{M2} we obtain a global extension of the examples constructed by Mooney in \cite{M2} with singular sets of dimension $2-\delta$ for any small $\delta>0$. 
\end{rem}

We can generalize the above example to an arbitrary dimension and codimension. Consider coordinates in an $n$ dimensional space as $(x,y,z) \in \mathbb{R}^{n-k}\times \mathbb{R}^{k-1}\times \mathbb{R}$ where $k\geq 3$ will be the codimension of the support of our obstacle. Using the same strategy as the example above, by taking the barrier as $\varphi = |x|^2/2+w_{1,k}(y,z)$ we get $n-k+1$ dimensional singularities arising away from the support of the obstacle, $\Sigma = \{(x,y,z)\;| \; |y|=0, z=\pm \rho_n/100 \text{ and } |x|\leq 1\}$.

We study an example where the geometry of the obstacle severely restricts the singular set of the solution even when $k\geq 3$. 

\begin{figure}[ht]
    \centering
    \includegraphics[width=0.5\textwidth]{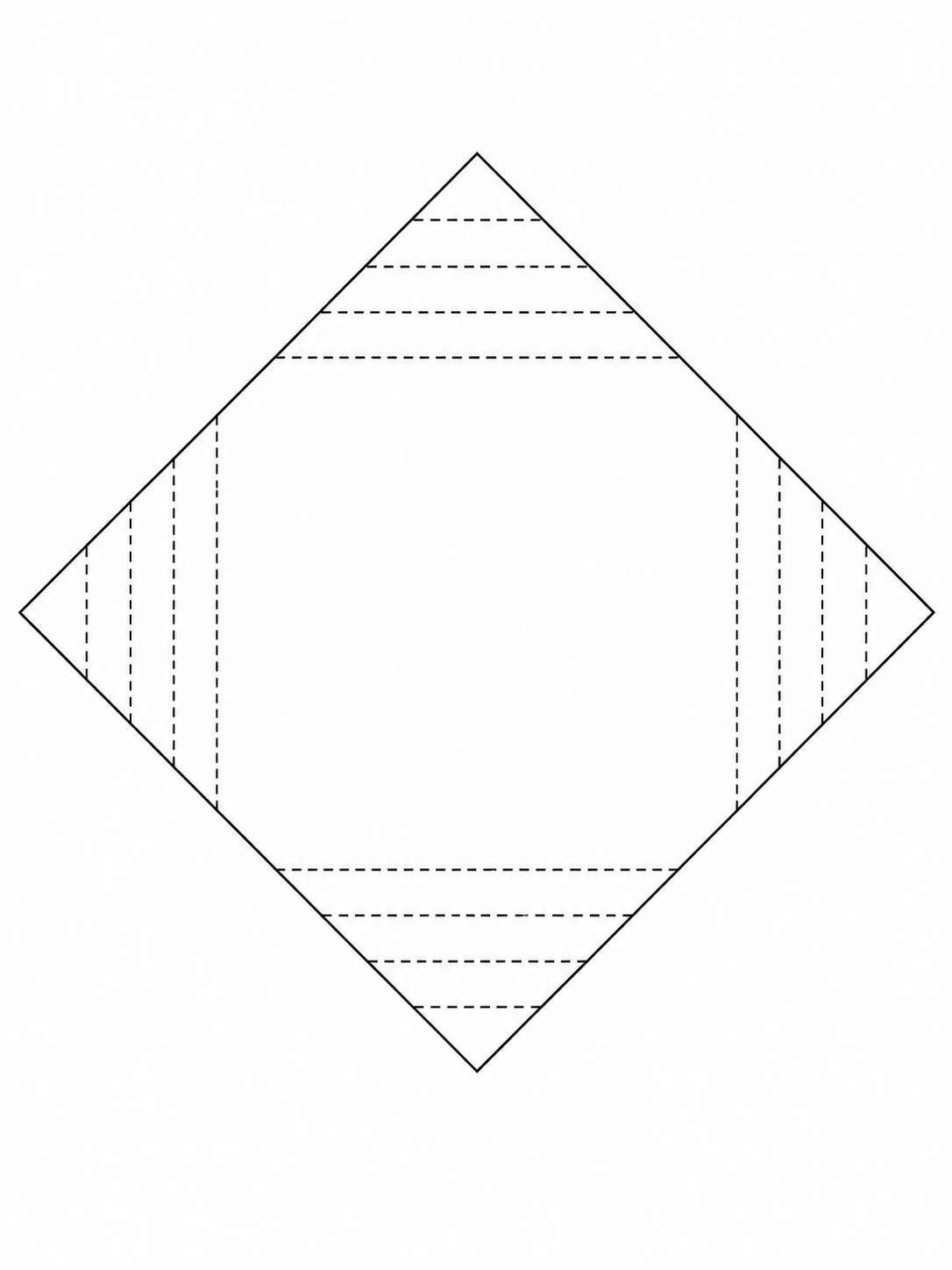}
    \caption{Interactions in a square}
    \label{sqinteract}
\end{figure}

\begin{ex}\label{square_4d}
    In dimension 4, denote the coordinates as $(x,y,z)\in \mathbb{R}\times\mathbb{R}^2\times\mathbb{R}$. Consider partial prescription of data (see Remark \ref{partial_data}) on a square on the $xz-$plane: $\Sigma = \{(x,0,z)\;|\;|x|+|z|= 1\}$. We solve the obstacle problem \ref{obstacle_problem} with appropriate parameters and $\alpha = 0.9$ as in the proofs of Theorems \ref{line}, \ref{polytope} and \ref{cross}. We further note that in this setting, $\min\{g\} = 0$ and this minimum is achieved at points $M = \{(\pm \frac{1}{2},\pm \frac{1}{2}),(\pm \frac{1}{2},-\pm \frac{1}{2})\}$. Here $g$ is the appropriately chosen obstacle. We claim that the singular set $S$ of the solution $u$ has the following property  $ \{(x,0,z)\;|\; 3/4 \geq \max\{|x|,|z|\} \geq  1/2\}\subseteq  S \subseteq  \{(x,0,z)\;|\; \max\{|x|,|z|\} \geq  1/2\}$. We argue similar to Theorem \ref{polytope}. By Theorem \ref{CL3_exterior}, $u \in C^{\infty}(\mathbb{R}^4\setminus \overline{U})$, for $U = \{(x,0,z)\;|\;|x|+|z| < 1\}$. Now assume a singularity arises away from the obstacle in $U$. Extreme points of a singularity of the form $\{u=L\}$, for some supporting hyperplane $L$, should lie on the set $\{u=g\}$. There are two cases that can arise:\\
    \begin{itemize}
        \item Case I: When two adjacent edges try to “communicate". This is possible by using appropriate translations of the barrier $\varphi_{\delta} = \delta x^2/2+C(\delta)\tilde{w}_{1,3}(y,z)$ from below for the obstacle problem (note that for a given $\delta$, we can choose $C(\delta)$ large enough to ensure $\det(D^2\varphi_\delta) \geq 1$). Here $\tilde{w}_{1,3}$ is a rescaled Pogorelov type barrier with domain $\{|z| \leq M\}$ for $1\ll M$. However this only guarantees contact with the obstacle on points when $|z| \leq 1/2$. Now due to symmetry, by rotating our barrier, we also achieve singularities on lines parallel to the $x-$axis. Therefore $S \supset \{(x,0,z)\;|\; 3/4 \geq \max\{|x|,|z|\} \geq  1/2\} $ (see Figure \ref{sqinteract}).

        \item Case II: We argue $\{u=L\} \cap \{(x,0,z)\;|\; \max\{|x|,|z|\} < 1/2\} = \emptyset$.  Since $u$ is linear on the singularity $\{u=L\}$, we get that $u \geq 0$ on the set $\{u=L\}$ by using the values of $u$ on the extreme points that lie in $\{u=g\}$. However $u = 0$ on $M$. Therefore $u\leq 0$ on $D = \{(x,0,z)\;|\; \max\{|x|,|z|\} < 1/2\}$. If $\{u=L\} \cap \{(x,0,z)\;|\; \max\{|x|,|z|\} < 1/2\} \neq \emptyset$, we get a contradiction unless $u=0$ on $L$. However, when $u=0$ on $L$, it is easy to see that $L$ restricted to the $xz-$plane has to be $0$. Moreover $L$ was a supporting hyperplane to $u$, therefore on the $xz-$plane $u \geq 0$. This shows that $\{u=0\} \supset \{(x,0,z)\;|\; \max\{|x|,|z|\} \leq 1/2\}$. This again is a contradiction because the dimension of a set $\{u=L\}$ for solutions to $\det(D^2u) \geq 1$ in $\mathbb{R}^4$ cannot be more than $1$ (see \cite[Proposition 2.3]{MR1}).
    \end{itemize}

\end{ex}
\begin{rem}
    In the above example \ref{square_4d}, the only reason a singularity couldn't arise in $\{(x,0,z)\;|\;  \max\{|x|,|z|\} \leq  1/2\}$ was because a $2$ dimensional failure of strict convexity of the form $\{u=L\}$ is not possible in $\mathbb{R}^4$ for solutions to $\det(D^2u) \geq 1$. In particular if we were in $\mathbb{R}^n$ for $n\geq 5$ there will be a singularity $\{u=0\}\supset \{(x,0,z)\;|\;  \max\{|x|,|z|\} \leq  1/2\}$ by using a Pogorelov type barrier $w_{2,n}$.
\end{rem}

  The above example shows that unless the geometry of the support of the obstacle is explicitly given, it is not possible in general to comment on the exact structure of the singular set. However we can comment when singularities may and may not arise in the context of the following theorem. To that end we state the following general result for a convex polytope $P$:

\begin{thm}[Volume obstruction to strict convexity]\label{vol_obs}
    For $R>0$, let $u_R$ be a solution to the following obstacle problem:
    \begin{equation}
        u_R = \sup \left\{v : v \in C\left(\overline{U}\right) \text{ convex},\, v \leq g \text{ in } U,\, v|_{\partial U} \leq \phi,\, Mv \geq \mu\right\}.
    \end{equation}
    Here $U = B_R(0)$, $g(x) = g_{n,\alpha}(\epsilon x)|_{P[n-k_1]}$, $\phi = W_n(R\textbf{e}_1)+10$, $\mu = 1dx$ and $\epsilon$ is a small constant depending on the convex polytope $P$. Then $u_R$ solves,
    $$\det(D^2u_R) = 1 + f\mathcal{H}^{n-k_1}|_{P[n-k_1]}$$
    for some $f\in L^{\infty}(P[n-k_1])$ in $B_R(0)$.
    For $k_2 < n/2$, consider a $k_2$ dimensional singularity away from $P[n-k_1]$ of the form $\{u_R=L\}$ for some supporting hyperplane $L$. Then the extreme points of $\{u_R=L\}$ always lie in the set $\{u_R=g\} \subset P[n-k_1]$. Moreover, if the singularity $\{u=L\}$ is contained in an $m$ dimensional face, where $m$ is the dimension of the convex hull of the $n-k_1$ faces on which the extreme points of $\{u_R = L\}$ lie, then $n> m+k_2$. 
\end{thm}  
\begin{proof}
    Since we know away from the support of the obstacle, $u_R$ solves $\det(D^2u_R)=1$, an extreme point of a singular set cannot lie away from the support of the obstacle. We rotate our coordinates such that $\mathbb{R}^m\times \mathbb{R}^{n-m}$ and the $m$ dimensional face lies on $\mathbb{R}^{m}\times \{0\}^{n-m}$. \\
    
    We prove the statement by contradiction, i.e., we assume $n\leq m+k_2$. Let $\{e_1,e_2,\cdots,e_{k_2+1},\cdots,e_{N}\}\subset \{u_R=L\}\cap \{u_R=g\}$ be the extreme points of the singularity $\{u_R=L\}$. Let $p$ denote the slope of the hyperplane $L$. Note that $S_h(e_i,p) = S_h(e_j,p)$ for all $i,j\in \{1,2,\cdots,N\}$. Since $u_R$ is strictly convex on $n-k_1$ dimensional faces, there can only be a single extreme point lying on an $n-k_1$ dimensional face. Now that we have ensured that $m$ is also the dimension of the convex hull of the $n-k_1$ faces on which the extreme points of $\{u_R = L\}$ lie, we know the behavior of $u_R$ in the section $S_h(e_1,h)$ along every direction in $\mathbb{R}^m\times \{0\}$. Therefore $S_h(e_1,p) \supset \cup_{i=1}^{N} (B_{\sqrt{h}}(e_i)\cap P[n-k_1])\cup \{u_R=L\}$ and volume of the section is at least $|S_h(e_1,p)| \geq l (\sqrt{h})^{m-k_2}\omega(h) h^{n-m-1}$, here $l$ is the $k_2$ dimensional Hausdorff measure of $\{u_R=L\}$. By using the fact that volume of a section of a solution to $\det(D^2u)\geq 1$ cannot be more than $\sim h^{n/2}$, we get $n > m+k_2$.
 
\end{proof}

While the statement of the above Theorem \ref{vol_obs} holds for any general convex polytope, in the specific case of a regular convex polytope, a very symmetric convex polytope, we can show that the above volume obstruction is the only possible obstruction to singularity formation in the context of the above result, i.e., the above Theorem is an if and only if statement for this symmetric case.

Let $P_n$ be a regular convex polytope in $\mathbb{R}^n$. Let $k_1,k_2,m$ and $\epsilon$ represent parameters as in Theorem \ref{vol_obs}. Then if $m+k_2< n$, we construct a $k_2$ dimensional singularity on an $m$ dimensional face. Consider $\rho$ such that $\epsilon \ll \rho \ll 1$. We rotate our coordinates such that $\mathbb{R}^n = \mathbb{R}^m\times \mathbb{R}^{n-m}$ and the $m$ dimensional face lies on $\mathbb{R}^{m}\times \{0\}^{n-m}$. Moreover we assume $P_n$ lies in half space $\mathbb{H} = \{(x_1,\cdots,x_m,\cdots,x_n)\;|\;x_i\geq 0 \;\forall i\in \{1,2,\cdots,m\}\}$. By symmetry we assume the barycentre of $P_n$ restricted to $\mathbb{R}^m\times \{0\}^{n-m}$ is the origin. Also consider a linear function $L$ that is zero along $\mathbb{R}^m\times \{0\}^{n-m}$, has $|\nabla L|=1/8$ and $P\subseteq \{L<0\}$. Consider the further splitting of coordinates as $\mathbb{R}^m\times \{0\}^{n-m} = \mathbb{R}^{m-k_2}\times\mathbb{R}^{k_2}\times\{0\}^{n-m}$. Now consider the following barrier, $$\varphi_{k_2,m} = \sum_{i=1}^{m-k_2}\frac{x_i^2}{2}+w_{k_2,n-m+k_2}+L.$$ 
Note that the sum appearing in $\varphi_{k_2,m}$ could be empty. By a small choice of $\epsilon$ and therefore $\rho$, we can also ensure $w_{k_2,n-m+k_2}+L \leq 0$ on $P_n$. This is possible because $w_{k_2,n-m+k_2}\sim C(n)dist(\cdot, \mathbb{R}^m\times \{0\}^{n-m})^{1+1/n}$, so the linear factor $L$ can ensure negativity for small $\epsilon$. Therefore $\varphi_{k_2,m} \in \mathcal{F}$, where $\mathcal{F}$ is,
$$\mathcal{F} = \left\{v : v \in C\left(\overline{U}\right) \text{ convex},\, v \leq g \text{ in } U,\, v|_{\partial U} \leq \phi,\, Mv \geq \mu\right\}$$
and the parameters are defined in Theorem \ref{vol_obs}. Moreover by an appropriate translation one can ensure $\varphi_{k_2,m}(x_1,\cdots,x_m,0,\cdots,0) = \sum_{i=1}^{m-k_2} \frac{x_i^{2}}{2}$ touches the obstacle. This is where symmetry becomes crucial. The degree of freedom available in our barrier $\varphi_{k_2,m}$, namely $m-k_2$, can be extremely small compared to the degree of freedom needed to touch the obstacle from below at sufficiently many points to ensure the development of a $k_2$ dimensional singularity independent of $n-k_1$ and $m-k_2$. We can now also consider barriers of the form $\varphi_{k_2,m,\delta} =  \sum_{i=1}^{m-k_2}\delta_i\frac{x_i^2}{2}+C(\delta)(w_{k_2,n-m+k_2}+L) $, for $0<\delta_i<1$ and appropriate $C(\delta)$ such that $\det(D^2\varphi_{k_2,m,\delta})\geq 1$. This ensures $\varphi_{k_2,m,\delta}\in \mathcal{F}$. By varying the parameters $\delta_i$ in $\varphi_{k_2,m,\delta}$, we can force the $k_2$ dimensional singularities to “integrate" over the $n-k_1$ dimensional obstacles in the sense of Example \ref{integrate}. This leads to a singular set of dimension higher than $k_2$.

\begin{rem}\label{sharp_sing}
    By using the modified obstacle $v_{\delta}$ constructed in \cite{M2} as our obstacle $g$ on the set $P[n-2]$ as above and the respective modified Pogorelov type barriers $w_\beta$ in place of the usual Pogorelov barriers, by repeating the arguments as above, we get a singular set with dimension $n-1-\delta$ for any small $\delta>0$. Note that the singular part of the Monge-Ampere measure will not be of the form $\mathcal{H}^{n-2-\delta}$.
\end{rem}

\section{Open Questions}\label{open}
In this final section we mention some open questions open questions for further study:
\begin{itemize}
    \item An interesting open problem in this direction is to analyze the fine structure of solutions near the singular set.  In particular, one may ask whether, in the presence of singular measures as our construction, the solution admits a precise asymptotic expansion near singular points, and how this expansion depends on the geometry of the singular set. When the singular set is just a Dirac mass, this was done in \cite{HTW}.
    \item In the examples presented in Section \ref{examples}, we construct optimal transport maps and gain some understanding of the associated singular structures in cases where one of the domains is non-convex. However, even in relatively simple configurations, it is not always clear how to analytically characterize the singular set, nor whether one can explicitly describe the corresponding optimal transport map. In general, it would be very interesting to develop methods for explicitly constructing optimal transport maps in such nontrivial geometric settings, as well as for identifying and describing their singular sets.
\end{itemize}

\section{Acknowledgments}
The authors would like to thank Professor Connor Mooney for several helpful discussions related to this paper. Arghya Rakshit would also like to thank Professor Robert McCann for discussions related to the optimal transport examples appearing in this paper. Aranya Sen was supported by NSF grant DMS-2143668 of Professor Connor Mooney.
%others for feedback


\begin{thebibliography}{9}

\bibitem{AC}
E.~Andriyanova, S.~Chen,
\newblock Boundary $C^{1,\alpha}$ regularity of potential functions in optimal transportation with quadratic cost,
\newblock \emph{Analysis \& PDE} \textbf{9} (2016), no.~6, 1483–1496..
\bibitem{B1}
Y.~Brenier,
\newblock D\'ecomposition polaire et r\'earrangement monotone des champs de vecteurs. (French)
\newblock \emph{C. R. Acad. Sci. Paris S\'er. I Math.} \textbf{305} (1987), no.~19, 805--808.

\bibitem{B2}
Y.~Brenier,
\newblock Polar factorization and monotone rearrangement of vector-valued functions,
\newblock \emph{Comm. Pure Appl. Math.} \textbf{44} (1991), no.~4, 375--417.

\bibitem{C0}
L.~A.~Caffarelli,
\newblock A localization property of viscosity solutions to the {Monge--Amp\`ere} equation and their strict convexity,
\newblock \emph{Ann. of Math.} \textbf{131} (1990), 129--134.

\bibitem{C1}
L.~A.~Caffarelli,
\newblock The regularity of mappings with a convex potential,
\newblock \emph{J. Amer. Math. Soc.} \textbf{5} (1992), 99--104.

\bibitem{C2}
L.~A.~Caffarelli,
\newblock Boundary regularity of maps with convex potentials II,
\newblock \emph{Ann. Math.} \textbf{144} (1996), no.~3, 453--496.



\bibitem{CL2}
L.~A.~Caffarelli, Y.~Y.~Li,
\newblock Some multi-valued solutions to the {Monge--Amp\`ere} equation,
\newblock \emph{Comm. Anal. Geom.} \textbf{14} (2006), 411--441.

\bibitem{CL3}
L.~A.~Caffarelli, Y.~Li,
\newblock An extension to a theorem of J\"orgens, Calabi, and Pogorelov,
\newblock \emph{Comm. Pure Appl. Math.} \textbf{56} (2003), 549--583.

\bibitem{CJLPR}
O.~Chodosh, V.~Jain, M.~Lindsey, L.~Panchev,  Y.~A.~Rubinstein,
\newblock On discontinuity of planar optimal transport maps,
\newblock \emph{J. Topol. Anal.} \textbf{7} (2015), no.~2, 239--260.

\bibitem{CT}
T.~C.~Collins, F.~Tong,
\newblock Boundary regularity of optimal transport maps on convex domains,
\newblock Preprint 2025, arXiv:2507.05395.

\bibitem{EG}
L.~C.~Evans, R.~F.~Gariepy,
\newblock Measure theory and fine properties of functions,
\newblock \emph{Textbooks in Mathematics}, Revised edition,
\newblock CRC Press, Boca Raton, FL, 2015.

\bibitem{real}
X.~Fern\'andez-Real,
\newblock The thin obstacle problem: a survey,
\newblock \emph{Publ. Mat.} \textbf{66} (2022), no.~1, 3--55.

\bibitem{H1}
R. ~Andreasson, J.~Hultgren,
\newblock Solvability of {Monge--Amp\`ere} equations and tropical affine structures on reflexive polytopes,
\newblock arXiv:2303.05276, 2023.

\bibitem{HJMM}
R. ~Andreasson, J.~Hultgren, M.~Jonsson, E.~Mazzon, N.~McCleerey,
\newblock Regularity of the solution to a real {Monge--Amp\`ere} equation on the boundary of a simplex,
\newblock \emph{Int. Math. Res. Not.} \textbf{2025} (2025), no.~3.


\bibitem{HTW}
G.~Huang, L.~Tang, X.-J.~Wang,
\newblock Regularity of free boundary for the {Monge--Amp\`ere} obstacle problem,
\newblock \emph{Duke Math. J.} \textbf{173} (2024), no.~12, 2259--2313.


\bibitem{JL}
M.~Jacobs, F.~L\'eger,
\newblock A fast approach to optimal transport: The back-and-forth method,
\newblock \emph{Numer. Math.} \textbf{146} (2020), no. 3, 513–544.




\bibitem{JX1}
T.~Jin, J.~Xiong,
\newblock Solutions of some Monge--Amp\`ere equations with isolated and line singularities,
\newblock \emph{Adv. Math.} \textbf{289} (2016), 114--141.

\bibitem{JTX1}
T.~Jin, X.~Tu,  J.~Xiong,
\newblock Extremal Alexandrov estimates: singularities, obstacles, and stability,
\newblock Preprint 2026, arXiv:2602.06468.

\bibitem{JTX2}
T.~Jin, X.~Tu, J.~Xiong,
\newblock Sharp global Alexandrov estimates and entire solutions of Monge--Amp`ere equations,
\newblock \emph{Math. Ann.} \textbf{395} (2026), no.~2, Article No.~41, 29 pp.


\bibitem{L1}
Y.~Li,
\newblock Strominger--Yau--Zaslow conjecture for Calabi--Yau hypersurfaces in the Fermat family,
\newblock \emph{Acta Math.} \textbf{229} (2022), no.~1, 1--53.

\bibitem{Lo1}
J.~C.~Loftin,
\newblock Singular semi-flat Calabi--Yau metrics on $S^2$,
\newblock \emph{Comm. Anal. Geom.} \textbf{13} (2005), 333--361.

\bibitem{LYZ1}
J.~Loftin, S.-T.~Yau, E.~Zaslow,
\newblock Affine manifolds, SYZ geometry and the ``Y''-vertex,
\newblock \emph{J. Differential Geom.} \textbf{71} (2005), 129--158.

\bibitem{McCann}
R.~J.~McCann,
\newblock A convexity principle for interacting gases,
\newblock \emph{Adv. Math.} \textbf{128} (1997), 153--179.

\bibitem{M1}
C.~Mooney,
\newblock Solutions to the {Monge--Amp\`ere} equation with polyhedral and {Y}-shaped singularities,
\newblock \emph{J. Geom. Anal.} \textbf{31} (2021), 9509--9526.

\bibitem{M2}
C.~Mooney,
\newblock Partial regularity for singular solutions to the Monge--Amp\`ere equation,
\newblock \emph{Comm. Pure Appl. Math.} \textbf{68} (2015), 1066--1084.

\bibitem{MR1}
C.~Mooney, A.~Rakshit,
\newblock Singular structures in solutions to the {Monge--Amp\`ere} equation with point masses,
\newblock \emph{Math. Eng.} \textbf{5} (2023), Paper No.~083, 11 pp.

\bibitem{MR2}
C.~Mooney, A.~Rakshit,
\newblock Sobolev regularity for optimal transport maps of non-convex planar domains,
\newblock \emph{SIAM J. Math. Anal.} \textbf{56} (2024), 4742--4758.


\bibitem{SY}
O.~Savin, H.~Yu,
\newblock Regularity of optimal transport between planar convex domains,
\newblock \emph{Duke Math. J.} \textbf{169} (2020), 1305--1327.









\end{thebibliography}
\end{document}